\documentclass[11pt]{article}
\usepackage{amsmath,amsthm,amsfonts,amssymb,bm,wasysym}
\usepackage{epsfig}
\usepackage[usenames]{color}
\usepackage{verbatim}
\usepackage{hyperref}
\usepackage{multicol}
\usepackage{comment}
\usepackage{float}
\usepackage{graphicx}
\usepackage{centernot}
\usepackage{tikz}
\usepackage{color}
\usepackage[]{algorithm2e}

\graphicspath{{./Pics}}

\usepackage[colorinlistoftodos]{todonotes}
\usepackage[normalem]{ulem}

\parskip \medskipamount
\newtheorem{theorem}{Theorem}[section]
\newtheorem{definition}[theorem]{Definition}

\numberwithin{equation}{section}
\newtheorem{lemma}[theorem]{Lemma}
\newtheorem{proposition}[theorem]{Proposition}
\newtheorem{corollary}[theorem]{Corollary}

\newtheorem{remark}[theorem]{Remark}

\numberwithin{equation}{section}

\def\N{\mathbb{N}}

\renewcommand{\phi}{\varphi}
\renewcommand{\epsilon}{\varepsilon}

\allowdisplaybreaks

\newcommand{\til}{\widetilde}

\newcommand{\pr}[1]{\mathbb{P}\!\left(#1\right)}

\newcommand{\prstart}[2]{\mathbb{P}_{#2}\!\left(#1\right)}

\newcommand{\tn}{|\kern-.1em|\kern-0.1em|}

\newcommand\be{\begin{equation}}
\newcommand\ee{\end{equation}}

\begin{document}

\title{Analyticity for rapidly determined properties of Poisson Galton--Watson trees}
\author{Yuval Peres\thanks{Email: {\tt yuval@yuvalperes.com}.}, Andrew Swan\thanks{University of Cambridge, Statistical Laboratory, DPMMS. Email: {\tt acks2@cam.ac.uk}.}}

	\tikzstyle{every node}=[circle, draw, fill=black!50,
	inner sep=0pt, minimum width=4pt]
	
	\maketitle
	\begin{abstract}
			Let $T_\lambda$ be a Galton--Watson tree with Poisson($\lambda$) offspring, and let $A$ be a tree property. In this paper, are concerned with the regularity of the function $\mathbb{P}_\lambda(A):= \mathbb{P}(T_\lambda \vdash A)$. We show that if a property $A$ can be uniformly approximated by a sequence of properties $A_k$, depending only on the first $k$ vertices in the breadth first exploration of the tree, with a bound in probability of $\mathbb{P}_\lambda(A\triangle A_k) \le Ce^{-ck}$ over an interval $I = (\lambda_0, \lambda_1)$, then $\mathbb{P}_\lambda(A)$ is real analytic in $\lambda$ for $\lambda \in I$.   We also present some applications of our results, particularly to properties that are not expressible in the first order language of trees.
	\end{abstract}

\section{Introduction}

Let $X_1, X_2,\ldots$ be a sequence of independent Poisson random variables of parameter $\lambda$. Set $\til{X} = (X_1, X_2, X_3, \dots)$ and construct a tree $T$ so that node $i$ has $X_i$ children, labelling the nodes from top to bottom and left to right, i.e., breadth first ordering (see Figure 1). We call the sequence $\til{X}$ the seed of the Poisson Galton--Watson tree $T$ with parameter~$\lambda$. Note that if the tree has a finite number $n$ of vertices then the values $X_j$ for $j > n$ are irrelevant.

Although the offspring distribution completely determines the law of $T$, it does not provide an immediate sense of the tree's structure. A more transparent structural description of $T$ is provided by tree property probabilities, i.e., for a given tree property $A$, what is the probability that $T$ has this property? For convenience, we will identify this event $T \vdash A$ with the property $A$ itself, defining
\begin{equation}
f_A(\lambda) := \prstart{T \vdash A}{\lambda},
\end{equation}
where we write $\prstart{\cdot}{\lambda}$ to indicate that the parameter of the Poisson distribution is $\lambda$.
In this paper we are interested in the regularity of $f_{\lambda}(A)$ as a function of $\lambda$ for certain choices of the tree property~$A$.

In essence, this is a question about phase transitions: loss of regularity in $\prstart{A}{\lambda}$ at a particular value of $\lambda$ is interpreted as phase transition in structure of $T_\lambda$, as `seen by' property $A$.
We illustrate this idea as follows. Consider the two events
\begin{equation}\label{key1}
\begin{split}
A_1 &= \left\{\text{The tree is infinite}\right\} = \left\{|T_\lambda| = \infty\right\},\\
A_2 &=  \left\{ \text{The root has exactly one child}\right\} =  \left\{X_1 =  1\right\}.
\end{split}
\end{equation}
As is well known (see, e.g., Prop.\ 5.4 in \cite{LyonsPeres}), the probability $f_{A_1}(\lambda)$ that $T_\lambda$ is infinite satisfies
\begin{equation}\label{survive}
f_{A_1}(\lambda) = 1 -\exp(-\lambda f_{A_1}(\lambda))
\end{equation}
Equivalently,
\begin{equation}\label{key2}
f_{A_1}(\lambda) = 1+\frac{W_0(-\lambda e^{-\lambda})}{\lambda},
\end{equation}
where $W_0(x)$ is the principle branch of the Lambert W function studied in \cite{Corless}, the unique real solution to
\begin{equation}\label{key3}
W_0(x)e^{W_0(x)} = x, \quad  W_0(x) \ge -1.
\end{equation}
This function $f_{A_1}(\lambda)=\prstart{A_1}{\lambda}$ is real analytic on $I_1 = (0,1)$ and on $I_2 = (1,\infty)$, but has a branch cut singularity at $\lambda = 1$ and so is not real analytic on any interval containing this point: the interpretation is that the size of a Poisson Galton--Watson tree undergoes a phase transition at $\lambda = 1$. On the other hand, the probability that the root node has exactly one child is
\begin{equation}\label{key4}
f_{A_2}(\lambda)= \lambda e^{-\lambda},
\end{equation}
which is a real analytic function over the entire domain $I = (0, \infty)$. From the perspective of $A_2$, there is no phase transition.

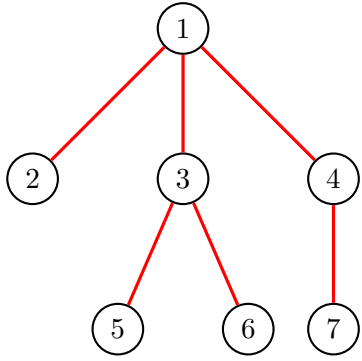
\begin{figure}
	\centering
	
	\begin{tikzpicture}
	\begin{scope}[every node/.style={circle,thick,draw}]
	\node (x1) at (3/1.5,6/1.5) {$1$};
	\node (x2) at (0,3/1.5) {$2$};
	\node (x3) at (3/1.5,3/1.5) {$3$};
	\node (x4) at (6/1.5,3/1.5) {$4$};
	\node (x5) at (1.7/1.5,0) {$5$};
	\node (x6) at (4.3/1.5,0) {$6$};
	\node (x7) at (6/1.5,0) {$7$};
	
	\end{scope}
	
	\begin{scope}[
	every node/.style={fill=white,circle},
	every edge/.style={draw=red,very thick}]
	\path [-] (x1) edge (x2);
	\path [-] (x1) edge (x3);
	\path [-] (x1) edge (x4);
	\path [-] (x3) edge (x5);
	\path [-] (x3) edge (x6);
	\path [-] (x4) edge (x7);
	\end{scope}
	\end{tikzpicture}
	\caption{The node labelling convention. The seed used to define this tree is $(3,0,2,1,0,0,0, \dots)$.}
\end{figure}

Recently, Podder and Spencer~\cite{Moumanti, Moumanti2} studied this question in the context of first order properties on the tree.
Informally speaking, a first order property can be expressed as a sentence in first order logic, which contains an infinite number of variables, the equality~``$=$'' relation, the binary parent relation $\pi(x,y)$ which is true if $y$ is the parent of $x$, the root symbol $R$, universal and existential quantifiers and the usual Boolean connectives.

In~\cite{Moumanti2}, Podder and Spencer used the Ehrenfeucht game for rooted trees and a contraction mapping theorem to prove the following:
\begin{theorem}\label{first}
	Let $A$ be a first order property. Then $f_A(\lambda)$ is a $C^\infty(0,\infty)$ function.
\end{theorem}
\noindent Our main result, Theorem \ref{thm:rapanal}, is an extension of this to a larger class of properties. We also improve the smoothness. Before stating our result, we introduce some notation and definitions.

The \textbf{k-truncated seed} $\til{X}^{(k)} = (X_1,\dots, X_k)$ is given by the first $k$ elements of the seed $\til{X}$.

\begin{definition}\rm{
	An event $A$ is called \emph{$k$-tautologically determined} if there exists a set $B \subseteq \mathbb{N}^k$ such that
	\begin{equation}\label{key5}
	A = \left\{\til{X}^{(k)} \in  B\right\}.
	\end{equation}
	}
\end{definition}
\begin{definition}\rm{
Let $0 \le \lambda_0 < \lambda_1 \le \infty$ and let $I = (\lambda_0, \lambda_1)$ be an interval.
	An event $A$ is called \emph{rapidly determined} over $I$, if for every $\lambda\in I$ there exist positive constants $c$ and $C$, $k_0\in \mathbb{N}$, and a sequence of~$k$-tautologically determined events $A_k$ such that for all $k\geq k_0$
	\begin{equation}\label{key6}
	\prstart{A\triangle A_k}{\lambda} \le Ce^{-c k}.
	\end{equation}
	}
\end{definition}

\begin{theorem}[{\cite[Theorem~6.6]{Moumanti2}}]
	Every first order property is rapidly determined over $(0,\infty)$.
\end{theorem}

We can now state our main result.

\begin{theorem}\label{thm:rapanal}
Let $0 \le \lambda_0 < \lambda_1 \le \infty$ and let $I = (\lambda_0, \lambda_1)$ be an interval. Suppose that the property~$A$ is rapidly determined over the interval~$I$. Then $f_A(\lambda)$ is a real analytic function on $I$. 
\end{theorem}
The conclusion of the theorem means that for every $\lambda \in I$, there exists $\delta>0$ so that the function $f_A(\lambda)$ can be extended to a complex analytic function $f_A(z)$ on the disc $D_\delta(\lambda) = \left\{z\; | \; |z - \lambda| \le \delta \right\}$.
Theorem \ref{thm:rapanal} improves on Theorem \ref{first} in two ways.  Firstly, we broaden the scope of applicability to the larger class of rapidly determined properties, and secondly, we improve the regularity from $C^\infty$ to real analytic.
The  collection of first order properties is countable, since every first order property is specified by a finite sequence from a countable alphabet. On the other hand, Proposition \ref{uncount} in Section 3 describes uncountably many rapidly determined properties.

\begin{corollary}
	Let $A$ be a first order property. Then $f_A(\lambda)$ is a real analytic function of $\lambda\in (0,\infty)$.
\end{corollary}

\begin{proposition}\label{pro:examples}
	Let
	\begin{align*}
	A_1&=\{ \text{there exists a node on an even level with exactly one child} \}\\
	A_2&=\{\text{there exists a node on a prime level with exactly two children}\}.
	\end{align*}
	Then $A_1$ and $A_1\cup A_2$ are both rapidly determined on $(0,\infty)$.
\end{proposition}

\begin{remark}\rm{
We note that neither $A_1$ nor $A_1\cup A_2$ are first order properties. This follows from a simple modification of \cite[Theorems~2.1.3 and~2.3.3]{Strangelogic}.
}	
\end{remark}

Unlike Podder and Spencer, our methods are not model theoretic in nature.  Instead, we take a more direct, complex analytic approach. It is similar in spirit to the route taken in \cite{Peres1991, Peres1992}, where the regularity of Lyapunov exponents for products of discrete random matrices was studied.

\section{Analyticity for rapidly determined properties}

In this section we prove Theorem~\ref{thm:rapanal}. We begin with a preliminary result.

\begin{lemma}\label{lem:extension}
Let $k\in \mathbb{N}$ and let $A$ be a $k$-tautologically determined event. Then $f_A(\lambda)$ may be analytically continued to an entire function $f_A(z)$.
\end{lemma}
\begin{proof}[\bf Proof]
By the assumption on $A$ there exists $B\subseteq  \mathbb{N}^k$ such that $A=\{\tilde{X}_\lambda^{(k)} \in  B\}$. Therefore we have
\begin{equation*}
\begin{split}
\prstart{A}{\lambda} = \pr{\tilde{X}_\lambda^{(k)} \in  B} &= \sum_{(m_1, \dots, m_k) \in B}\prod_{i = 1}^k e^{-\lambda} \frac{\lambda^{m_i}}{m_i!}\\
&= e^{-k\lambda} \sum_{(m_1, \dots, m_k) \in B} \frac{\lambda^{m_1+\dots+m_k}}{m_1!\dots m_k!}= e^{-k\lambda} \sum_{n = 0}^\infty a_n \frac{\lambda^n}{n!},\\
\end{split}
\end{equation*}
where
\begin{equation}\label{key7}
0 \le a_n = \sum_{\substack{(m_1, \dots, m_k)\in B\\ m_1+\dots+m_k = n}} {n \choose m_1, \dots, m_k} \le \sum_{\substack{(m_1, \dots, m_k)\in \N^k\\ m_1+\dots+m_k = n}}  {n \choose m_1, \dots, m_k} = k^n.
\end{equation}
Since
\begin{equation}\label{key8}
\lim_{n \rightarrow \infty} \left|\frac{a_n}{n!}\right|^{\frac{1}{n}} \le \lim_{n\rightarrow \infty}\left|\frac{k^n}{n!}\right|^{\frac{1}{n}} = 0,
\end{equation}
it follows that $\prstart{A}{\lambda}$ may be analytically continued to an entire function $\prstart{A}{z}$.
\end{proof}

\begin{proof}[\bf Proof of Theorem~\ref{thm:rapanal}]
Let $\lambda \in I$.
	Since $A$ is rapidly determined over the interval $I$, there exist constants $c$ and $C$, $k_0\in \N$ and a sequence of $k$-tautologically determined events $(A_k)$ so that for all $k\geq k_0$
	\begin{align}\label{eq:rapidly}
			\prstart{A\triangle A_k}{\lambda}\leq Ce^{-ck}.	
	\end{align}
    From this it then follows that
    \begin{align*}
    \prstart{A}{\lambda} = \lim_{k\to\infty} \prstart{A_k}{\lambda} = \lim_{n\to\infty} \sum_{k=1}^{n}\prstart{A_k\setminus A_{k-1}}{\lambda}.
    \end{align*}
	From Lemma~\ref{lem:extension} we get that $f_{A_k}(\lambda)=\prstart{A_k}{\lambda}$ can be extended to a complex analytic function over $\mathbb{C}$ that we denote $f_{A_k}(z)$. In order to establish that $f_A(\lambda)$ can also be extended to an analytic function in some neighbourhood of $\lambda\in I$, it suffices to show that for every $\lambda\in I$ there exist positive constants~$c_1$ and~$c_2$ and $\delta>0$ such that for all $z\in D_\delta(\lambda) = \{z\in \mathbb{C}: |z-\lambda|\leq \delta\}$ we have
	\begin{align}\label{eq:uniform}
		|f_{A_k\setminus A_{k-1}}(z)| \leq c_1 e^{-c_2k}.
	\end{align}
	Indeed, this will then imply that $f_{A_n}(z)$ converges uniformly to a function denoted $f_{A}(z)$,  which will also be analytic on $D_\delta(\lambda)$.
	
	We start by showing that for all $k\in \N$, if $\Gamma$ is a $k$-tautologically determined event with $$\Gamma=\{\til{X}^{(k)}\in M\}$$ with $M\subseteq \N^k$, then we have the following: for all $\epsilon>0$ there exists $\delta=\delta(\epsilon,\lambda)$ so that for every~$z\in D_\delta(\lambda)$ the analytic continuation of $f_\Gamma(\lambda)$ satisfies
	\begin{align}\label{eq:goal}
			|f_{\Gamma}(z)| \leq \sum_{\ell=0}^{\infty} (1+\epsilon)^{\ell} \sum_{\substack{(m_1,\ldots, m_k)\in M \\ \sum_{i\leq k}m_i=\ell }} \prod_{i=1}^{k} e^{-\lambda} \frac{\lambda^{m_i}}{m_i!}.
	\end{align}
	We have
	\begin{align}\label{eq:complexprob}
		f_\Gamma(z) = \sum_{\ell=0}^{\infty}\sum_{\substack{(m_1,\ldots, m_k)\in M \\ \sum_{i\leq k}m_i=\ell }} \prod_{i=1}^{k} e^{-z} \frac{z^{m_i}}{m_i!}.
	\end{align}
	For $z\in D_\delta(\lambda)$ with $\delta \le \min \left\{\frac{\lambda \epsilon}{2}, \log{\frac{1+\epsilon}{1+\epsilon/2}}\right\}$, we have for every $r\geq 0$
\begin{equation*}
\left| \frac{z^{r} e^{-z}}{r!} \right| \le   \left|1+\frac{\delta}{\lambda}\right|^r e^{\delta}\frac{\lambda^{r} e^{-\lambda}}{r!}\leq \left(1+\epsilon\right)^{r} \frac{\lambda^{r}e^{-\lambda}}{r!}.
\end{equation*}
	Using this for all $k$ terms of the product appearing in~\eqref{eq:complexprob} now proves~\eqref{eq:goal}.
	
	Since the event $A_k\setminus A_{k-1}$ is $k$-tautologically determined, we let $B\subseteq \N^k$ be such that $$A_k\setminus A_{k-1}=\{\til{X}^{(k)}\in B\}.$$ We  can now apply~\eqref{eq:goal} and get for $\epsilon$ and $\delta$ as above
	\begin{align}\label{eq:long}
	\begin{split}
		|f_{A_k\setminus A_{k-1}}(z)|&\leq \sum_{\ell \leq [3k\lambda]}(1+\epsilon)^{\ell} \sum_{\substack{(m_1,\ldots, m_k)\in B \\ \sum_{i\leq k}m_i=\ell }} \prod_{i=1}^{k} e^{-\lambda} \frac{\lambda^{m_i}}{m_i!} + \sum_{\ell > [3k\lambda]}(1+\epsilon)^{\ell} \sum_{\substack{(m_1,\ldots, m_k)\in B \\ \sum_{i\leq k}m_i=\ell }} \prod_{i=1}^{k} e^{-\lambda} \frac{\lambda^{m_i}}{m_i!} \\
		&\leq (1+\epsilon)^{3k\lambda} \cdot \prstart{A_k\setminus A_{k-1}}{\lambda} + \sum_{\ell >[3k\lambda]} (1+\epsilon)^{\ell} \cdot \prstart{\sum_{i=1}^{k}X_i=\ell}{\lambda}.
		\end{split}
	\end{align}
	Since $\sum_{i=1}^{k}X_i$ has the Poisson distribution with parameter $k\lambda$, it follows that there exists a positive constant $c_1$ so that for $\ell>[3k\lambda]$
	\[
	\prstart{\sum_{i=1}^{k}X_i=\ell}{\lambda}\leq e^{-c_1 \ell}.
	\]
	From~\eqref{eq:rapidly} we get that there exist positive constants $c_2$ and $c_3$ so that for all $k$
	\[
	\prstart{A_k\setminus A_{k-1}}{\lambda}\leq c_2e^{-c_3k}.
	\]
	Taking $\epsilon$ sufficiently small and using the two bounds above into~\eqref{eq:long} we obtain for positive constants~$c_4$ and $c_5$
	\begin{align*}
		|f_{A_k\setminus A_{k-1}}(z)| \leq c_4 e^{-c_5 k}
	\end{align*}
	and this concludes the proof of~\eqref{eq:uniform} and also the proof of the theorem.
	\end{proof}

\section{Examples of rapidly determined properties}

In this section we provide some examples of rapidly determined properties to demonstrate the applicability of Theorem~\ref{thm:rapanal}.

We start by showing that when the tree is subcritical, every property is rapidly determined.
For a tree $T$ we write $|T|$ for the total number of vertices of $T$.

\begin{proposition}
Let $0 \le \lambda_0 < \lambda_1 \le 1$. Then every property $A$ is rapidly determined on the interval $I=(\lambda_0,\lambda_1)$.
\end{proposition}

\begin{proof}[\bf Proof]
Let $A_k=A\cap \{|T|<k\}$. Since $A_k$ is a $k$-tautologically determined event, it suffices to show that for every $\lambda<1$ there exist positive constants $c$ and $C$ so that for all $k$
\begin{align}\label{eq:rapidsub}
\prstart{A\triangle A_k}{\lambda} \leq Ce^{-ck}.
\end{align}
We now have
\[
\prstart{A\triangle A_k}{\lambda}= \prstart{A\setminus A_k}{\lambda} \leq \prstart{|T|\geq k}{\lambda}\leq \prstart{\sum_{i=1}^{k}X_i\geq k}{\lambda}.
\]
Using that $\sum_{i=1}^{k}X_i$ has the Poisson distribution with parameter $k\lambda$ and $\lambda<1$ proves~\eqref{eq:rapidsub} (See, e.g., Appendix A in \cite{AlonSpencer}),  and this concludes the proof.
\end{proof}

\begin{remark}\rm{
One interpretation of the proposition above is that Poisson Galton--Watson trees do not exhibit a phase transition in \emph{any} property over the interval $I = (0,1)$.}
 \end{remark}

\begin{lemma}\label{lem:even}
	Let $E_k$ be the set of nodes amongst the first $k$ which lie on an even level. Then for every $\lambda \in (0,\infty)$ there exists a positive constant $c$ so that
	\begin{equation}\label{key9}
	\prstart{|E_k| \le \left\lfloor\frac{k}{2\lambda + 1}\right\rfloor,|T| \ge k }{\lambda }\le e^{-ck}.
	\end{equation}
\end{lemma}

\begin{proof}[\bf Proof]
	Let $E_k$ ($O_k$) be the set of nodes amongst the first $k$ which lie on an even (odd) level. On the event $\left\{|T| \ge k\right\}$ all the first $k$ nodes exist, and hence
	\begin{equation}\label{keyodd}
	|O_k| + |E_k| = k.
	\end{equation}
	Let $Y_1,\ldots$ be i.i.d.\ Poisson$(\lambda)$ random variables. If the $i$-th vertex on an even level exists, then attach to it $Y_i$ children.
	From~\eqref{keyodd} we then get that on the event $\left\{|T| \ge k\right\}$
	\[
	\sum_{i=1}^{|E_k|} Y_i \geq |O_k| = k-|E_k|.
	\]
	Set $n=\lfloor k/(2\lambda +1)\rfloor$. There exists a positive constant $c$ so that
	\begin{align*}
		\prstart{|E_k|\leq n, |T|\geq k}{\lambda} \leq \prstart{\sum_{i=1}^{n}Y_i\geq k-n}{\lambda}\leq e^{-ck}
	\end{align*}
	 and this concludes the proof.
\end{proof}

Next we prove a more general statement than the one given in~Proposition~\ref{pro:examples}. As noted in the Introduction, this statement implies that there are uncountably many rapidly determined properties.

In the following, if $F \subset \mathbb{N}$ is a set of levels, we say a node lies on an \emph{$F$-level} if the level of the node is contained in $F$.
\begin{proposition}\label{uncount}
The event
\[
A = \left\{ \text{there exists a node on an even level with exactly one child}\right\}
\] is rapidly determined on the interval $I = (\lambda_0, \lambda_1)$ for any $0\le\lambda_0 < \lambda_1 \le \infty$.
	Moreover, if $F \subset \N$ is any set of levels, and $B$ is the event:
	\[
	B = \left\{ \text{there exists a node on an $F$-level with exactly two children}\right\}
	\]
	 then $A\cup B$ is a rapidly determined event.
\end{proposition}

\begin{proof}[\bf Proof]
Let $E_k$ and $F_k$ be the sets of nodes among the first $k$ which lie on an even/$F$-level, respectively. As in the proof of Lemma~\ref{lem:even}, let $Y_1,\ldots$ be i.i.d.\ Poisson$(\lambda)$ random variables, representing the number of children that are attached to the $i$-th vertex on an even level. We now define the event $A_k$ (resp. $B_k$) that in $E_k$ ($F_k$) there exists a node with exactly one (two) child(ren), i.e.,
\[
A_k = \bigcup_{i = 1}^{|E_k|} \left\{Y_i = 1\right\}, \qquad B_k = \bigcup_{i = 1}^{|F_k|} \left\{Y_i = 2\right\}
\]
Set $n=\lfloor k/(2\lambda +1)\rfloor$. We now have
\begin{align*}
	\prstart{(A\cup B)\triangle (A_k \cup B_k)}{\lambda} &\le \prstart{(A\cup B)\setminus (A_k\cup B_k), |T_\lambda|\geq k}{\lambda} \leq  \prstart{\bigcap_{i=1}^{|E_k|}\{Y_i\neq 1\}, |T_\lambda|\geq k}{\lambda} \\
	&\leq \prstart{|E_k|\leq n, |T_\lambda|\geq k}{\lambda} + \prstart{\bigcap_{i=1}^{n}\{Y_i\neq 1\}}{\lambda}\\
	&\leq e^{-ck} + (1-\lambda e^{-\lambda})^n \leq c_1 e^{-c_2k}
\end{align*}
for positive constants $c,c_1$ and $c_2$, where in the last inequality we have used Lemma~\ref{lem:even}.
\end{proof}

\bibliography{biblio}

\begin{thebibliography}{1}

\bibitem{AlonSpencer}
N.~Alon and J.~H. Spencer.
\newblock {\em The probabilistic method}.
\newblock Wiley Series in Discrete Mathematics and Optimization. John Wiley \&
  Sons, Inc., Hoboken, NJ, fourth edition, 2016.

\bibitem{Corless}
R.~M. Corless, G.~H. Gonnet, D.~E. Hare, D.~J. Jeffrey, and D.~E. Knuth.
\newblock On the {L}ambert {W} function.
\newblock {\em Advances in Computational mathematics}, 5(1):329--359, 1996.

\bibitem{LyonsPeres}
R.~Lyons and Y.~Peres.
\newblock {\em Probability on Trees and Networks}, volume~42 of {\em Cambridge
  Series in Statistical and Probabilistic Mathematics}.
\newblock Cambridge University Press, New York, 2016.
\newblock Available at \url{http://pages.iu.edu/~rdlyons/}.

\bibitem{Peres1991}
Y.~Peres.
\newblock {Analytic dependence of Lyapunov exponents on transition
  probabilities}.
\newblock In {\em Lyapunov Exponents, Proceedings of a Conference held in
  Oberwolfach, May 28 - June 2, 1990}, volume 1486, pages 64--80. Springer,
  1991.

\bibitem{Peres1992}
Y.~Peres.
\newblock {Domains of analytic continuation for the top Lyapunov exponent}.
\newblock {\em Annales de l'Institut Henri Poincar\'{e} (B) Probabilit\'{e}s et
  Statistiques}, 28(1):131--148, 1992.

\bibitem{Moumanti}
M.~Podder and J.~Spencer.
\newblock First order probabilities for {G}alton-{W}atson trees.
\newblock In {\em A journey through discrete mathematics}, pages 711--734.
  Springer, Cham, 2017.

\bibitem{Moumanti2}
M.~Podder and J.~Spencer.
\newblock Galton-{W}atson probability contraction.
\newblock {\em Electron. Commun. Probab.}, 22:Paper No. 20, 16, 2017.

\bibitem{Strangelogic}
J.~Spencer.
\newblock {\em The strange logic of random graphs}, volume~22 of {\em
  Algorithms and Combinatorics}.
\newblock Springer-Verlag, Berlin, 2001.

\end{thebibliography}
\bibliographystyle{abbrv}

\end{document}